\documentclass[a4paper,11pt]{amsart}
\usepackage{amsthm,amssymb,mathtools,amsfonts,latexsym,rawfonts,amsmath, mathrsfs, lscape}
\usepackage{stmaryrd,tikz,pgfplots}
\usepackage{marginnote}  
\usepackage[backref=page]{hyperref}
\usepackage[top=1.5in, bottom=0.9in, left=0.9in, right=0.9in]{geometry}

\usepackage[textsize=tiny]{todonotes}
\newcommand{\Marginpar}[1]{\marginpar{\tiny{#1}}}
\newcommand{\Note}[1]{{\par\noindent\hrulefill\par\tiny{#1}\par\noindent\hrulefill\par}}
\newcommand{\Detail}[1]{{#1}}
\renewcommand{\Marginpar}[1]{}
\renewcommand{\Note}[1]{}
\renewcommand{\Detail}[1]{}
\renewcommand*{\backref}[1]{}
\renewcommand*{\backrefalt}[4]{%
    \ifcase #1 (Not cited.)%
    \or        (Cited on page~#2.)%
    \else      (Cited on pages~#2.)%
    \fi}

\usepackage{hyperref}
\usepackage[all]{xy}
\usepackage{amsmath,amsfonts}	% Bunch of math support
\usepackage{amsthm,latexsym,amssymb}
\usepackage{tensor}

\usepackage{array, tabularx}

\usepackage{booktabs} %for nice tables
\usepackage{siunitx}

\usepackage{tikz-cd}

\newtheorem{thm}{Theorem}
\newtheorem*{thm*}{Theorem}
\newtheorem{prop}[thm]{Proposition}

\newtheorem*{lem*}{Lemma}
\newtheorem{cor}[thm]{Corollary}

\theoremstyle{definition}
\newtheorem{defn}[thm]{Definition}
\newtheorem*{defn*}{Definition}

\newtheorem{rem}[thm]{Remark}

\renewcommand{\[}{\begin{equation*}}
\renewcommand{\]}{\end{equation*}}

\begin{document}
\parskip1mm

\title[Canonical almost-K\"ahler metrics]{Canonical almost-K\"ahler metrics dual to general plane-fronted wave Lorentzian metrics}

\author{Mehdi Lejmi}
\address{Department of Mathematics, Bronx Community College of CUNY, Bronx, NY 10453, USA.}
\email{mehdi.lejmi@bcc.cuny.edu}

\author{Xi Sisi Shen}
\address{Department of Mathematics, Columbia University, New York, NY 10027, USA.}
\email{xss@math.columbia.edu}

\thanks{The authors are warmly grateful to Amir Babak Aazami for giving his insights on pp-wave spacetimes and his comments on the note. The authors are also thankful to Daniele Angella, Giuseppe Barbaro, Abdellah Lahdili and Ali Maalaoui for very useful discussions. The first author is supported by the Simons Foundation Grant \#636075. }



\maketitle
\begin{abstract}
In the compact setting, Aazami and Ream~\cite{Aazami:2022th} proved that Riemannian metrics dual to a class of Lorentzian metrics, called (compact) general plane-fronted waves, are almost-K\"ahler. In this note, we explain how to construct extremal and second-Chern--Einstein non-K\"ahler almost-K\"ahler metrics dual to those general plane-fronted waves.  
  \end{abstract}
  
  \section{introduction}
  
Let $(M,g)$ be a Riemannian manifold of dimension $2n$. The Riemannian metric $g$ is called almost-K\"ahler if there is a $g$-orthogonal almost-complex structure $J$ on $M$ such that the induced $2$-form $\omega$ defined by $\omega(\cdot,\cdot):=g(J\cdot,\cdot)$ is a symplectic form. When $J$ is integrable, $g$ is a K\"ahler metric. On an almost-K\"ahler manifold $(M,g,\omega,J)$, the $1$-parameter family of canonical Hermitian connections introduced by Gauduchon in~\cite{MR1456265} all coincide and are equal
to the canonical Hermitian connection $\nabla$ which is the unique connection preserving the almost-K\"ahler structure $(g,\omega,J)$ and whose torsion is the Nijenhuis tensor of $J$~\cite{MR66733,MR1456265} (in particular $\nabla$ is the Chern connection on $(M,g,\omega,J)$). More explicitly, $$\nabla_XY=D^g_XY-\frac{1}{2}J\left(D^g_XJ\right)Y,$$
where $D^g$ is the Levi--Civita connection with respect to $g$ and $X,Y$ are vector fields on $M.$ The Hermitian scalar curvature
$s^H$ of the almost-K\"ahler metric $g$ is obtained by taking twice the trace of $R^\nabla$ the curvature of $\nabla$ with respect to $\omega$.

Donaldson~\cite{MR1622945} showed how one can use a formal framework of the geometric invariant theory~\cite{MR1304906} in order to define natural representatives of almost-K\"ahler metrics called extremal almost-K\"ahler metrics~\cite{MR1969266,MR2747965,MR2661166}. An almost-K\"ahler metric $g$ is extremal if the symplectic gradient of $s^H$ is an infinitesimal isometry of $J$ (or equivalently a Killing vector field with respect to $g$). Extremal almost-K\"ahler metrics include almost-K\"ahler metrics with constant Hermitian
scalar curvature and Calabi extremal K\"ahler metrics~\cite{MR645743}. For more results about extremal non-K\"ahler almost-K\"ahler metrics,
we refer the reader to~\cite{MR2747965,MR2661166,MR3331165,MR2795448,MR2988734,MR4039808,MR4140765,MR3967373,MR4118148,MR4437352,MR4140765,Cahen:2022tm}. In addition to extremal almost-K\"ahler metrics, other possible canonical almost-K\"ahler metrics are the so-called second-Chern--Einstein metrics. Namely, an almost-K\"ahler metric is second-Chern--Einstein if $R^\nabla(\omega)$ is proportional to $\omega$ (here $s^H$ is not necessarily constant). When $J$ is integrable, second-Chern--Einstein almost-K\"ahler metrics are K\"ahler--Einstein. On complex manifolds, Hermitian metrics satisfying the second-Chern--Einstein condition were studied for instance in~\cite{MR571563,MR1477631,MR2781927,MR2795448,MR2925476,MR3869430,MR4125707,Angella:2020vt,Barbaro:2022vn}

Recently, Aazami and Ream~\cite{Aazami:2022th} demonstrate a one-to-one correspondence between certain almost-K\"ahler metrics
and a class of Lorentzian metrics called pp-wave spacetimes (see, e.g.,~\cite[Chapter 13]{MR619853} and~\cite{Aazami:2022th} and the references therein). In the compact setting, a generalization of pp-wave spacetimes is given by general plane-fronted wave Lorentzian metrics~\cite{MR1971289}. It turns out that Riemannian metrics dual to general plane-fronted wave Lorentzian metrics are almost-K\"ahler~\cite[Theorem 2]{Aazami:2022th}.
\begin{thm}~\cite[Theorem 2]{Aazami:2022th}
Let $(M,g_M)$ be a closed almost-K\"ahler manifold. Let $\varphi,\theta$ denote the standard angular coordinates on $\mathbb{S}^1\times\mathbb{S}^1$ and $\textbf{H}$ a function on $\mathbb{S}^1\times M$ independent of the first angular coordinate $\varphi.$
Consider the (compact) general plane-fronted wave Lorentzian metric $h$ defined on $\mathbb{S}^1\times\mathbb{S}^1\times M:$
$$h:=2\,d\varphi\,d\theta+\textbf{H}\, d\theta^2+g_M.$$
With respect to the vector field
$$T:=\frac{1}{2}\left(\textbf{H}+1 \right)\partial_\varphi-\partial_\theta,$$
let $g$ denote the Riemannian metric dual to $h$:
\begin{equation}\label{form-dual}
g:=h+2\,T^\flat\otimes T^\flat,
\end{equation}
where $\flat$ the $h$-Lorentzian dual of $T$. Then $g$ is an almost-K\"ahler metric on $\mathbb{S}^1\times\mathbb{S}^1\times M$
which is not a warped product.
\end{thm}
In the present paper, after Preliminaries, we explain in Section~\ref{Sec2} how to construct extremal almost-K\"ahler metrics of the form~(\ref{form-dual}) dual to (compact) general plane-fronted wave Lorentzian metrics and we obtain the following
 \begin{thm*} {   \upshape (see~Theorem~\ref{thm-extremal})}
 Let $(M,g_M)$ be a closed extremal almost-K\"ahler manifold with a non-constant Hermitian scalar curvature $s^H_M$. Suppose that $\textbf{H}=s^H_M$. Then, on $\mathbb{S}^1\times\mathbb{S}^1\times M$, the metric $g$, of the form~(\ref{form-dual}) dual to a general plane-fronted wave Lorentzian metric, is an extremal almost-K\"ahler metric.
 \end{thm*}
 We also remark that when $(M,g_M)$ has a constant Hermitian scalar curvature then the metric $g$, of the form~(\ref{form-dual}) dual to a general plane-fronted wave Lorentzian metric has also a constant Hermitian scalar curvature (see~Theorem~\ref{thm-constant}).
 
 In Section~\ref{Sec3}, we study the second-Chern--Einstein condition and we obtain for instance an infinite family of second-Chern--Einstein non-K\"ahler almost-K\"ahler metrics on $\mathbb{S}^1\times\mathbb{S}^1\times \mathbb{S}^2$ of the form~(\ref{form-dual}) dual to a general plane-fronted waves (see also Theorem~\ref{hirzebruch} in dimension 6). 
 
 \begin{thm*}{\upshape (see Theorem~\ref{2-chern-dim-4}})
Let $\Sigma$ be a compact Riemann surface equipped with a Riemannian metric $g_\Sigma$. Suppose that $\textbf{H}$ is a non-constant function on $\Sigma$. Then, the metric $g$, of the form~(\ref{form-dual}) dual to a general plane-fronted wave Lorentzian metric, is a second-Chern--Einstein non-K\"ahler almost-K\"ahler metric on $\mathbb{S}^1\times\mathbb{S}^1\times \Sigma$ if and only if
\begin{equation}\label{eq-H-intro}
\|grad^g \textbf{H}\|^2=2s^H_\Sigma,
\end{equation}
where $grad^g \textbf{H}$ is the gradient of $\textbf{H}$ with respect to the metric $g$ and $s^H_\Sigma$ is twice the Gaussian curvature of $(\Sigma,g_\Sigma)$. In particular, $\Sigma$ is isomorphic to the sphere $\mathbb{S}^2.$ Moreover, with a suitable normalization of $\textbf{H}$, a metric $g_\Sigma$ satisfying Equation~(\ref{eq-H-intro}) always exists.
\end{thm*}

  \section{Preliminaries}\label{Sec1}
  
  Let $(M,g)$ be a Riemannian manifold of dimension $2n$. The metric $g$ is almost-K\"ahler if there is a $g$-orthogonal almost-complex structure $J$ i.e. $g(J\cdot,J\cdot)=g(\cdot,\cdot)$ such that the $2$-form $\omega(\cdot,\cdot):=g(J\cdot,\cdot)$ is $d$-closed
 i.e. $d\omega=0$ where $d$ is the exterior derivative. The almost-complex structure $J$ is integrable if and only if the Nijenhuis tensor
 $N$ defined by $$4N(X,Y):=[JX,JY]-[X,Y]-J[JX,Y]-J[X,JY],$$
 vanishes~\cite{MR88770} (here $X,Y$ are vector fields on $M$). 
 For an almost-K\"ahler metric, we have the following
 \begin{equation}\label{AK-relation}
 g\left((D^g_XJ)Y,Z\right)=2g\left( JX,N(Y,Z)\right),
 \end{equation}
 where $D^g$ is the Levi--Civita connection of $g.$ In particular, $D^g$ preserves $J$ if and only if $g$ is K\"ahler. 
 The canonical Hermitian connection $\nabla$ is defined by
 $$\nabla_XY=D^g_XY-\frac{1}{2}J\left(D^g_XJ\right)Y.$$ The connection $\nabla$ preserves the almost-K\"ahler structure $(g,\omega,J)$ and its torsion is given by the Nijenhuis tensor $N.$
 We denote by $R^\nabla$ the curvature of $\nabla$ and we use the convention $R^\nabla_{X,Y}=\nabla_{[X,Y]}-[\nabla_X,\nabla_Y].$
The first-Chern--Ricci form $\rho^\nabla$ is defined by
$$\rho^\nabla(X,Y)=\frac{1}{2}\displaystyle\sum_{i=1}^{2n}g\left(R^\nabla_{X,Y}e_i,Je_i  \right),$$
where $\{e_1,e_2=Je_1,\ldots,e_{2n-1},e_{2n}=Je_{2n-1}\}$ is a $J$-adapted $g$-orthogonal local frame of the tangent bundle.
The $2$-form $\rho^\nabla$ is $d$-closed and it is a representative of $2\pi c_1(TM,J)$ the first Chern class of $M.$ 
On the other hand, the second-Chern--Ricci form $r$ is defined by
$$r(X,Y)=\frac{1}{2}\displaystyle\sum_{i=1}^{2n}g\left(R^\nabla_{e_i,Je_i}X,Y  \right).$$
Equivalently, $r=R^\nabla(\omega).$ The $2$-form $r$ is $J$-invariant but not closed in general.
\begin{defn}
An almost-K\"ahler metric is second-Chern--Einstein if $$r=\lambda\,\omega,$$
for some function $\lambda.$
\end{defn}
We would like to express the difference between $\rho^\nabla$ and $r$ in terms of the Nijenhuis tensor $N$. Denote by $\rho^\star$ the $\star$-Ricci form defined as $$\rho^\star(X,Y)=R^g(\omega)(X,Y)=\frac{1}{2}\sum_{i=1}^{2n}g\left(R^g_{e_i,Je_i}X,Y \right),$$
 where $R^g$ is the curvature of the Levi--Civita connection $D^g.$ Then, from~\cite{MR1782093}, we have
 $$\rho^\nabla(X,Y)=\rho^\star(X,Y)-\frac{1}{4}\sum_{i=1}^{2n}g\left(J(D^g_XJ)(D^g_YJ)e_i,e_i       \right)$$
On the other hand from~\cite[Proposition 1]{Barbaro:2022vn} and Equation(\ref{AK-relation}), it is easy to see that
$$r(X,Y)=\left(\rho^\star\right)^{J,+}(X,Y)+\frac{1}{4}\sum_{i=1}^{2n}g\left((D^g_{e_i}J)(D^g_{Je_i}J)X,Y       \right),$$
where $\left(\cdot\right)^{J,+}$ denotes the $J$-invariant part.
We obtain then
$$r(X,Y)=\left(\rho^\nabla\right)^{J,+}(X,Y)-\frac{1}{4}\sum_{i=1}^{2n}g\left((D^g_XJ)(D^g_YJ)e_i,Je_i       \right)+\frac{1}{4}\sum_{i=1}^{2n}g\left((D^g_{e_i}J)(D^g_{Je_i}J)X,Y       \right),$$
 
From Equation(\ref{AK-relation}), we get the following formula
\begin{cor}\label{2-chern-formula}
Let $(M,g)$ be an almost-K\"ahler manifold. Then,
\begin{align*}
r(X,Y)&=\left(\rho^\nabla\right)^{J,+}(X,Y)+\sum_{i,k=1}^{2n}g\left(JX,N(e_i,e_k)      \right)g\left(Y,N(e_i,e_k)      \right)\\
&+\sum_{i,k=1}^{2n}g\left(N(X,e_k),e_i       \right)g\left(N(Y,e_k),Je_i       \right),
\end{align*}
where $\{e_1,e_2=Je_1,\ldots,e_{2n-1},e_{2n}=Je_{2n-1}\}$ is a $J$-adapted $g$-orthonormal local frame of the tangent bundle.
\end{cor}

The trace of $\rho^\nabla$ and $r$ with respect to $\omega$ is equal to the Hermitian scalar curvature
$$s^H=\sum_{i,k=1}^{2n}\rho^\nabla(e_i,Je_i)=\sum_{i=1}^{2n}r(e_i,Je_i).$$
On the other hand the Riemannian scalar curvature $s^g$ is defined as $$s^g=\displaystyle\sum_{i,j=1}^ng\left(R^g_{e_i,e_j}e_i,e_j\right),$$
and coincides with $s^H$ when $g$ is K\"ahler.

An almost-K\"ahler metric $g$ is extremal~\cite{MR2747965} if the symplectic gradient $grad_\omega s^H=J\,grad^gs^H$ of the Hermitian scalar curvature $s^H$ is a Killing vector field with respect to $g$ (here $grad^g$ denotes the $g$-Riemannian gradient).

  \section{Extremal almost-K\"ahler metrics dual to general plane-fronted wave Lorentzian metrics}\label{Sec2}

 Let $(M,g_M,\omega_M,J_M)$ be a closed almost-K\"ahler manifold of dimension $2n$. On the manifold $M$, we use the Darboux coordinates $\{z_1,\ldots,z_n,t_1,\ldots,t_n\}$ defined on an open set. The symplectic form $\omega_M$ has the form 
 \begin{equation*}
 \omega_M=\displaystyle\sum_{i=1}^ndz_i\wedge dt_i.
 \end{equation*}
% On the other hand, the metric $g_M$ has the form
% \begin{equation*}
% g_M=\displaystyle\sum_{i,j=1}^nG_{ij}(z,t)\,dz_i\otimes dz_j+H_{ij}(z,t)\,dt_i\otimes dt_j+P_{ij}(z,t)\,\left(dz_i\otimes dt_j+dt_j\otimes dz_i\right).
% \end{equation*}
 Let $\varphi,\theta$ denote the standard angular coordinates on $\mathbb{S}^1\times\mathbb{S}^1$. Let $\textbf{H}$ be an arbitrary smooth function
 on $\mathbb{S}^1\times M$ that is independent of the angular coordinate $\varphi.$ Then, on $\mathbb{S}^1\times\mathbb{S}^1\times M$,
 we consider the almost-K\"ahler metric $g$ of the form~(\ref{form-dual}) dual to a (compact) general plane-fronted wave Lorentzian metric. Then, the metric $g$ can be expressed as 
 \begin{equation}\label{wave-metric}
 g=g_M+\textbf{H}\left(d\varphi\otimes d\theta+ d\theta\otimes d\varphi\right)+\frac{1}{2}(1+\textbf{H}^2)\,d\theta^2+2\,d\varphi^2.
 \end{equation}
 
 The almost-K\"ahler metric $g$ is compatible with the following symplectic form $\omega$ (i.e. $(g,\omega)$ is an almost-K\"ahler structure)
 \begin{equation}\label{symp-form}
 \omega=\omega_M+d\theta\wedge d\varphi.
 \end{equation}
 The metric $g$ is K\"ahler if and only if $\textbf{H}$ is constant on $M$ and $g_M$ is K\"ahler~\cite[Theorem 2]{Aazami:2022th}.
 
 We would like to compute the first-Chern--Ricci form $\rho^\nabla$ and the Hermitian scalar curvature $s^H$ on the almost-K\"ahler manifold
 $(\mathbb{S}^1\times\mathbb{S}^1\times M,g,\omega)$. We apply the formulae computed in~\cite{MR2747965} using Darboux coordinates. 
 
 \begin{prop}\label{chern-ricci}
 Let $\rho^\nabla$ be the first-Chern--Ricci form of the almost-K\"ahler manifold $(\mathbb{S}^1\times\mathbb{S}^1\times M,g,\omega)$. Then,
  \begin{eqnarray*}
  \rho^\nabla=\rho^\nabla_M+\frac{1}{2}\displaystyle\sum_{i=1}^n\textbf{H}_{\theta z_i}d\theta\wedge dz_i+\frac{1}{2}\displaystyle\sum_{i=1}^n\textbf{H}_{\theta t_i}d\theta\wedge dt_i,
  \end{eqnarray*}
  where $\rho^\nabla_M$ is the first-Chern--Ricci form of $(M,g_M,\omega_M,J_M)$. Here we use the notation $\textbf{H}_{\theta z_i}=\frac{\partial^2 \textbf{H}}{\partial\theta\partial z_i},$ etc.
 \end{prop}
 
 \begin{cor}\label{scalar-cur}
 Let $s^H$ be the Hermitian scalar curvature of the almost-K\"ahler manifold $(\mathbb{S}^1\times\mathbb{S}^1\times M,g,\omega)$. Then
 $$s^H=s^H_M,$$
 where $s^H_M$ is the Hermitian scalar curvature of $(M,g_M,\omega_M,J_M)$.
 \end{cor}
 \begin{proof}
 This is a direct consequence of Proposition~\ref{chern-ricci}.
 \end{proof}
 
As a consequence of Corollary~\ref{scalar-cur}, we obtain the following
 
 \begin{thm}\label{thm-constant}
 Let $(M,g_M,\omega_M,J_M)$ be a closed almost-K\"ahler manifold of constant Hermitian scalar curvature. Then, on $\mathbb{S}^1\times\mathbb{S}^1\times M$, the metric $g$, of the form~(\ref{form-dual}) dual to a general plane-fronted wave Lorentzian metric,
has a constant Hermitian scalar curvature.
 \end{thm}
 
 We also get the following as a consequence of Proposition~\ref{chern-ricci}.

 \begin{cor}
Let $(M,g_M,\omega_M,J_M)$ be a closed non-K\"ahler almost-K\"ahler manifold with a zero first-Chern--Ricci form $\rho^\nabla_M=0$. Suppose that $\textbf{H}$ is a function on $M$. Then, the metric $g$, of the form~(\ref{form-dual}) dual to a general plane-fronted wave Lorentzian metric, defines on $\mathbb{S}^1\times\mathbb{S}^1\times M$ a non-K\"ahler almost-K\"ahler metric with vanishing first-Chern--Ricci form $\rho^\nabla=0.$
 \end{cor}
 
  Now, we would like to discuss the construction on $\mathbb{S}^1\times\mathbb{S}^1\times M$ of extremal almost-K\"ahler metrics with non-constant Hermitian scalar curvature. 
  \begin{thm}\label{thm-extremal}
 Let $(M,g_M,\omega_M,J_M)$ be a closed extremal almost-K\"ahler manifold with a non-constant Hermitian scalar curvature $s^H_M$. Suppose that $\textbf{H}=s^H_M$. Then, on $\mathbb{S}^1\times\mathbb{S}^1\times M$, the metric $g$, of the form~(\ref{form-dual}) dual to a general plane-fronted wave Lorentzian metric, is an extremal almost-K\"ahler metric.
 \end{thm}
 \begin{proof}
 We need to prove that $grad_\omega s^H$ is a Killing vector field of $g$, where $\omega$ is given by~(\ref{symp-form}). From Corollary~\ref{scalar-cur}, it follows that $s^H=s^H_M$ and so $grad_\omega s^H=grad_\omega s_M^H.$ Consider the $J_M$-adapted $g_M$-orthonormal local frame of the tangent bundle of $M$ $\{X_1,J_MX_1,\ldots,X_n,J_MX_n\}$. We define as in~\cite[Theorem 2]{Aazami:2022th}
 $$T=\frac{1}{2}(\textbf{H}+1)\partial_\varphi-\partial_\theta.$$ 

The almost-complex structure $J_M$ is then extended to an almost-complex structure $J$ on $\mathbb{S}^1\times\mathbb{S}^1\times M$ by defining
 $$JT=\frac{1}{2}(\textbf{H}-1)\partial_\varphi-\partial_\theta.$$ The frame $\{X_1,JX_1,\ldots,X_n,JX_n,T,JT\}$ defines on $\mathbb{S}^1\times\mathbb{S}^1\times M$ a $J$-adapted $g$-orthonormal local frame.
The $g$-Riemannian duals of $T$ and $JT$ are the $1$-forms 
\begin{align}
T^\flat&=\frac{1}{2}(\textbf{H}-1)d\theta+d\varphi,\label{T-dual}\\
(JT)^\flat&=-\frac{1}{2}(\textbf{H}+1)d\theta-d\varphi.\label{JT-dual}
\end{align} 
 Then, the Riemannian metric $g$ is given by
 \begin{equation}\label{g-express-T}
 g=g_M+T^\flat\otimes T^\flat+(JT)^\flat\otimes (JT)^\flat.
 \end{equation}
We prove now that $\mathfrak{L}_{grad_\omega s_M^H}g=0$, where $\mathfrak{L}$ is the Lie derviative. First, for any two vectors $X_i,X_j$ in $\{X_1,JX_1,\ldots,X_n,JX_n\}$, we have
\begin{equation*}
\left(\mathfrak{L}_{grad_\omega s_M^H}g\right)(X_i,X_j)=\left((\mathfrak{L}_{grad_\omega s_M^H}T^\flat)\odot T^\flat+\mathfrak{L}_{grad_\omega s_M^H}(JT)^\flat\odot (JT)^\flat\right)(X_i,X_j)=0,
\end{equation*}
 using~(\ref{g-express-T}) and because $grad_\omega s_M^H$ is a Killing vector field of $g_M$. For the same reasons 
 $$\left(\mathfrak{L}_{grad_\omega s_M^H}g\right)(T,X_i)=\left(\mathfrak{L}_{grad_\omega s_M^H}g\right)(JT,X_i)=0.$$
 Moreover using Cartan formula and the fact that $g(grad_\omega s_M^H,T)=g(grad_\omega s_M^H,JT)=0$ we get
 \begin{align*}
\left(\mathfrak{L}_{grad_\omega s_M^H}g\right)(T,T)&=\left((\mathfrak{L}_{grad_\omega s_M^H}T^\flat)\odot T^\flat+\mathfrak{L}_{grad_\omega s_M^H}(JT)^\flat\odot (JT)^\flat\right)(T,T),\\
&=2\left(\mathfrak{L}_{grad_\omega s_M^H}T^\flat\right)(T),\\
&=2\,dT^\flat(grad_\omega s_M^H,T),\\
&=\left(d\textbf{H}\wedge d\theta\right)(grad_\omega s_M^H,T),\\
&=-d\textbf{H}\wedge \left(T^\flat+ (JT)^\flat \right)(grad_\omega s_M^H,T),\\
&=-ds^H_M\left(grad_\omega s_M^H\right)=0.
\end{align*}
Here we use the assumption that $\textbf{H}=s^H_M$ and the fact that $-d\theta=T^\flat+ (JT)^\flat$ which follows from~(\ref{T-dual}) and~(\ref{JT-dual}).
In a similar way,
 \begin{align*}
\left(\mathfrak{L}_{grad_\omega s_M^H}g\right)(JT,JT)&=\left((\mathfrak{L}_{grad_\omega s_M^H}T^\flat)\odot T^\flat+\mathfrak{L}_{grad_\omega s_M^H}(JT)^\flat\odot (JT)^\flat\right)(JT,JT),\\
&=2\left(\mathfrak{L}_{grad_\omega s_M^H}(JT)^\flat\right)(JT),\\
&=2\,d(JT)^\flat(grad_\omega s_M^H,JT),\\
&=-\left(d\textbf{H}\wedge d\theta\right)(grad_\omega s_M^H,JT),\\
&=d\textbf{H}\wedge \left(T^\flat+ (JT)^\flat \right)(grad_\omega s_M^H,JT),\\
&=ds^H_M\left(grad_\omega s_M^H\right)=0.
\end{align*}
 \begin{align*}
\left(\mathfrak{L}_{grad_\omega s_M^H}g\right)(T,JT)&=\left((\mathfrak{L}_{grad_\omega s_M^H}T^\flat)\odot T^\flat+\mathfrak{L}_{grad_\omega s_M^H}(JT)^\flat\odot (JT)^\flat\right)(T,JT),\\
&=\left(\mathfrak{L}_{grad_\omega s_M^H}T^\flat\right)(JT)+\left(\mathfrak{L}_{grad_\omega s_M^H}(JT)^\flat\right)(T),\\
&=dT^\flat(grad_\omega s_M^H,JT)+d(JT)^\flat(grad_\omega s_M^H,T),\\
&=\frac{1}{2}\left(d\textbf{H}\wedge d\theta\right)(grad_\omega s_M^H,JT)-\frac{1}{2}\left(d\textbf{H}\wedge d\theta\right)(grad_\omega s_M^H,T),\\
&=-\frac{1}{2}d\textbf{H}\wedge \left(T^\flat+ (JT)^\flat \right)(grad_\omega s_M^H,JT)+\frac{1}{2}d\textbf{H}\wedge \left(T^\flat+ (JT)^\flat \right)(grad_\omega s_M^H,T)=0.\\
\end{align*}
The theorem follows.
\end{proof}
 
  \section{Second-Chern--Einstein almost-K\"ahler metrics dual to general plane-fronted wave Lorentzian metrics}\label{Sec3}
  \subsection{The four-dimensional case}
 Let $\Sigma$ be a compact Riemann surface with a Riemannian metric $g_\Sigma$ and a complex structure $J_\Sigma$. In this section, we would like to investigate the existence of second-Chern--Einstein
 almost-K\"ahler metrics of the form~(\ref{form-dual}) dual to a general plane-fronted wave Lorentzian metric on $\mathbb{S}^1\times\mathbb{S}^1\times \Sigma$. On $\Sigma$, we use the isothermal coordinates $\{x,y\}$ so that $$g_\Sigma=e^{2u}\left(dx\otimes dx+dy\otimes dy \right),$$
 for some function $u=u(x,y)$.

We have the $g_\Sigma$-unit vectors $$X=e^{-u}\partial_x,\,J_\Sigma X=e^{-u}\partial_y.$$
%with $g_\Sigma$-Riemannian dual $1$-forms
%$$X^\flat=e^{u}\,dx,\,\quad(J_\Sigma X)^\flat=e^{u}\,dy.$$
 
 The complex structure $J_\Sigma$ is extended to an almost-complex structure $J$ on $\mathbb{S}^1\times\mathbb{S}^1\times \Sigma$ by considering the vector fields  
  \begin{align*}
  T&=\frac{1}{2}(\textbf{H}+1)\partial_\varphi-\partial_\theta,\\
  JT&=\frac{1}{2}(\textbf{H}-1)\partial_\varphi-\partial_\theta.
  \end{align*}
 On the almost-K\"ahler manifold $(\mathbb{S}^1\times\mathbb{S}^1\times \Sigma,g,J)$, where $g$ is of the form~(\ref{form-dual}), we can express the second-Chern-Ricci form $r$ in terms of the first-Chern--Ricci form $\rho^\nabla$ using Corollary~\ref{2-chern-formula} or~\cite[Proposition 24]{Barbaro:2022vn}
 \begin{equation}\label{second-chern-dim4}
 r=\left(\rho^\nabla \right)^{J,+}-\frac{1}{4}\|N\|^2\omega+4\,\left(N(X,T)\right)^\flat\wedge\left(JN(X,T)\right)^\flat,
 \end{equation}
 where $\left(\cdot\right)^{J,+}$ denotes the $J$-invariant part, $N$ is the Nijenhuis tensor of $J$, $\|N\|^2=8\|N(X,T)\|^2$ and $\flat$ denotes the $g$-Riemannian dual.

Now, we compute $r$ using~(\ref{second-chern-dim4}). First, we have
 \begin{align*}
 [X,T]=[X,JT]&=\frac{X(\textbf{H})}{2}\left(T-JT \right),\\
  [JX,T]=[JX,JT]&=\frac{JX(\textbf{H})}{2}\left(T-JT \right).
 \end{align*}
 Hence,
 \begin{align*}
 4N(X,T)&=[JX,JT]-[X,T]-J[X,JT]-J[JX,T],\\
 &=\frac{JX(\textbf{H})}{2}\left(T-JT \right)-\frac{X(\textbf{H})}{2}\left(T-JT \right)-\frac{X(\textbf{H})}{2}\left(JT+T \right)-\frac{JX(\textbf{H})}{2}\left(JT+T \right),\\
 &=-X(\textbf{H})T-JX(\textbf{H}) JT.
 \end{align*}
  We obtain that
  \begin{align}\label{nijenhuis}
  4\,\left(N(X,T)\right)^\flat\wedge\left(JN(X,T)\right)^\flat&=\frac{\left(X(\textbf{H})\right)^2+\left(JX(\textbf{H})\right)^2}{4}T^\flat\wedge(JT)^\flat,\nonumber\\
  &=\frac{\left(X(\textbf{H})\right)^2+\left(JX(\textbf{H})\right)^2}{4}d\theta\wedge d\varphi.
  \end{align}
 On the other hand, we compute $\rho^\nabla$ using isothermal coordinates on $\Sigma$ and we get the following
 \begin{prop}\label{first-chern-4}
On $(\mathbb{S}^1\times\mathbb{S}^1\times \Sigma,g,J)$, we have
$$\rho^\nabla=-dJdu+\frac{1}{2}\textbf{H}_{x\theta}\, d\theta\wedge dx+\frac{1}{2}\textbf{H}_{y\theta}\, d\theta\wedge dy,$$ 
where $\textbf{H}_{x\theta}=\frac{\partial^2 \textbf{H} }{\partial x\partial \theta},$ etc.
 \end{prop}
   
  \begin{proof}
 We follow computations done in~\cite{MR1921552,MR1894944}. Two independant $J$-anti-invariant forms $\phi,J\phi=-\phi(J\cdot,\cdot)$ of square norms $2$ are given by
\begin{align*}
\phi&=X^\flat\wedge T^\flat-(JX)^\flat\wedge  (JT)^\flat ,\\
&={e^u}\left(dx\wedge\left(\frac{1}{2}(\textbf{H}-1)d\theta+d\varphi\right) +dy\wedge \left(\frac{1}{2}(\textbf{H}+1)d\theta+d\varphi\right) \right),
\end{align*}
and
\begin{align*}
J\phi&=X^\flat\wedge (JT)^\flat +(JX)^\flat\wedge T^\flat,\\
&={e^u}\left(dy\wedge\left(\frac{1}{2}(\textbf{H}-1)d\theta+d\varphi\right) -dx\wedge \left(\frac{1}{2}(\textbf{H}+1)d\theta+d\varphi\right) \right).
\end{align*}
We then have $$d\phi=\tau_\phi\wedge\phi,\quad d(J\phi)=\tau_{J\phi}\wedge J\phi,$$
for some $1$-forms $\tau_\phi,\tau_{J\phi}.$ We obtain that $$\tau_\phi=\left(u_x+\frac{\textbf{H}_x}{2}-\frac{\textbf{H}_y}{2}\right)\,dx+\left(u_y+\frac{\textbf{H}_x}{2}-\frac{\textbf{H}_y}{2}\right)\,dy=du+\left(\frac{\textbf{H}_x}{2}-\frac{\textbf{H}_y}{2}\right)(dx+dy),$$
and
$$\tau_{J\phi}=\left(u_x-\frac{\textbf{H}_x}{2}-\frac{\textbf{H}_y}{2}\right)\,dx+\left(u_y+\frac{\textbf{H}_x}{2}+\frac{\textbf{H}_y}{2}\right)\,dy=du-\left(\frac{\textbf{H}_x}{2}+\frac{\textbf{H}_y}{2}\right)(dx-dy).$$
The first-Chern--Ricci form is given by
$$\rho^\nabla=-\frac{1}{2}d\left(J\tau_\phi+J\tau_{J\phi}\right).$$
The proposition follows.
 \end{proof}
 Combining~(\ref{second-chern-dim4}) and~(\ref{nijenhuis}) and Proposition~\ref{first-chern-4}, we get the following
 \begin{cor}\label{second-chern-4}
 On $(\mathbb{S}^1\times\mathbb{S}^1\times \Sigma,g,J)$, the second-Chern-Ricci form is given by
 \begin{align*}
 r&=-dJdu+\frac{\textbf{H}_{x\theta}-\textbf{H}\textbf{H}_{y\theta}}{4}d\theta\wedge dx+\frac{\textbf{H}_{y\theta}-\textbf{H}\textbf{H}_{x\theta}}{4}d\theta\wedge dy-\frac{\textbf{H}_{y\theta}}{2}d\varphi\wedge dx+\frac{\textbf{H}_{x\theta}}{2}d\varphi\wedge dy\\
&-\frac{\left(X(\textbf{H})\right)^2+\left(JX(\textbf{H})\right)^2}{8}\omega+\frac{\left(X(\textbf{H})\right)^2+\left(JX(\textbf{H})\right)^2}{4}d\theta\wedge d\varphi.
 \end{align*}
 \end{cor}
 As a consequence of Corollary~\ref{second-chern-4}, we get the following theorem
 \begin{thm}\label{2-chern-dim-4}
Let $\Sigma$ be a compact Riemann surface equipped with a Riemannian metric $g_\Sigma$. Suppose that $\textbf{H}$ is a non-constant function on $\Sigma$. Then, the metric $g$, of the form~(\ref{form-dual}) dual to a general plane-fronted wave Lorentzian metric, is a second-Chern--Einstein non-K\"ahler almost-K\"ahler metric on $\mathbb{S}^1\times\mathbb{S}^1\times \Sigma$ if and only if
\begin{equation}\label{eq-H}
\|grad^g \textbf{H}\|^2=2s^H_\Sigma,
\end{equation}
where $grad^g \textbf{H}$ is the gradient of $\textbf{H}$ with respect to the metric $g$ and $s^H_\Sigma$ is twice the Gaussian curvature of $(\Sigma,g_\Sigma)$. In particular, $\Sigma$ is isomorphic to the sphere $\mathbb{S}^2.$ Moreover, for a suitable normalization of $\textbf{H}$, a metric $g_\Sigma$ satisfying Equation~(\ref{eq-H}) always exists.
\end{thm}
\begin{proof} 
We will only prove the last statement. Let $g_0$ be the metric on $\Sigma$ with Gaussian curvature equal to $1.$ Suppose that $g_\Sigma=e^{2f}g_0,$ for some function $f.$ Then, Equation~(\ref{eq-H}) is equivalent to $$\|grad^{g_0} \textbf{H}\|^2_0=2e^{2f}s^H_\Sigma,$$ where the gradient and the norm are with respect to $g_0$. So given a function $\textbf{H}$ on $\Sigma$, we need to find a solution $f$ to $$\|grad^{g_0} \textbf{H}\|^2_0=4+4\Delta^{g_0}f,$$
where $\Delta^{g_0}$ is the Laplacian with respect to $g_0.$ We conclude that given a non-constant $\textbf{H}$, there exists a metric $g_\Sigma$ satisfying
Equation~(\ref{eq-H}) if $\textbf{H}$ is normalized so that $$\displaystyle\int_\Sigma \|grad^{g_0} \textbf{H}\|^2_0=4V_{g_0}, $$
where $V_{g_0}$ is the total volume with respect to $g_0.$  
\end{proof}

\begin{rem}
If $\textbf{H}$ is a constant function and $s^H_\Sigma=0$, in particular $\Sigma$ is isomorphic to the torus $T^2$, then we get a K\"ahler Ricci flat metric
on $\mathbb{S}^1\times\mathbb{S}^1\times \Sigma.$
 \end{rem}

Theorem~\ref{2-chern-dim-4} provides examples of non-K\"ahler almost-K\"ahler metrics $g$ that are second-Chern--Einstein and weakly $\star$-Einstein (i.e. the $\star$-Ricci form $\rho^\star$ is proportional to $\omega$) with a $J$-anti-invariant Riemannian Ricci tensor (in particular the Riemannian scalar curvature $s^g=0$) and a $J$-invariant first-Chern--Ricci form with $c_1(\mathbb{S}^1\times\mathbb{S}^1\times \mathbb{S}^2,J)=0$. We remark that neither $s^\star=2s^H=\|grad^g \textbf{H}\|^2$ can be constant nor $g$ can be Riemannian Einstein unless the metric $g$ is K\"ahler~\cite{MR1317012} (see also~\cite{MR1604803,MR1108474,MR1607545,MR1727533}) (here $s^\star$ is twice the trace of $\star$-Ricci form with respect to $\omega$).

   \subsection{Six-dimensional example}

%We suppose that $(M,\omega_M,g_M,J_M)$ is a B\'erard-Bergery standard cohomogeneity one K\"ahler manifold with two singular orbits. 

We would like to construct here a six-dimensional example of second-Chern--Einstein non-K\"ahler almost-K\"ahler metric of the form~(\ref{form-dual}) dual to a general plane-fronted wave Lorentzian metric.
We denote by $M=\mathbb{P}\left(\mathcal{O} _{\mathbb{CP}^1}\oplus\mathcal{O} _{\mathbb{CP}^1}(1)\right)$ the first-Hirzebruch surface. On $M,$ we consider the $U(2)$-invariant K\"ahler metric $g_M$ given by~\cite{MR727843}
 $$g_M=h^2\left( e^1\otimes e^1+e^2\otimes e^2\right)+h^2{h^\prime}^2e^3\otimes e^3+dt^2,$$
 where $e^1,e^2,e^3$ are the invariant $1$-forms on $\mathbb{S}^3$ dual to the vectors $X,Y,Z$ satisfying $[X,Y]=2V,[Y,V]=2X,[V,X]=2Y$ and $t$ is a coordinate transverse to the $U(2)$-orbits. The function $h=h(t)$ is a positive function on the interval $(0,l)$ and satisfy the boundary conditions
 \begin{equation}\label{conditions}
 h(0)>0,h(l)>0,h^{\prime\prime}(0)=\frac{1}{h(0)},h^{\prime\prime}(l)=-\frac{1}{h(l)},h^{(2p+1)}(0)=h^{(2p+1)}(l)=0,\,\forall \,p\geq 0.
 \end{equation}
 Let $\{E_1=\frac{1}{h}X,E_2=\frac{1}{h}Y,E_3=\frac{1}{hh^\prime}V,E_4=\frac{\partial}{\partial t}\}$ be a local $g_M$-orthonormal basis of the tangent bundle of $M$ and we consider the $g_M$-orthogonal complex structure $J_M$ defined by $$E_2=J_ME_1,\quad E_4=J_ME_3.$$
% In particular, $g_M$ is compatible with the K\"ahler form $$\omega_M=E_1^\flat\wedge E_2^\flat+E_3^\flat\wedge E_4^\flat,$$
% where  $\flat$ is $g_M$-Riemannian dual via the metric $g_M.$ 
% 
 On $\mathbb{S}^1\times\mathbb{S}^1\times M$, we consider the almost-K\"ahler metric $g$ of the form~(\ref{form-dual}) dual to a general plane-fronted wave Lorentzian metric. The complex structure $J_M$ is extended to
 an almost-complex structure $J$ on $\mathbb{S}^1\times\mathbb{S}^1\times M$ by introducing $T,JT$ defined by
  \begin{align*}
  T&=\frac{1}{2}(\textbf{H}+1)\partial_\varphi-\partial_\theta,\\
  JT&=\frac{1}{2}(\textbf{H}-1)\partial_\varphi-\partial_\theta.
  \end{align*}
 Suppose that $\textbf{H}$ is a function on $M$ invariant under the action of $U(2)$ so $\textbf{H}=\textbf{H}(t)$. Since $\textbf{H}$ is a function on $M$ then it follows from Proposition~\ref{chern-ricci} that on $(\mathbb{S}^1\times\mathbb{S}^1\times M,g,J)$ $$\rho^\nabla=\rho^\nabla_M,$$
 where $\rho^\nabla_M$ is the first-Chern--Ricci form of $(M,g_M,J_M).$ 
\begin{prop}\label{2-chern-cohomogeneity}
On $(\mathbb{S}^1\times\mathbb{S}^1\times M,g,J)$, we have
$$r=\rho^\nabla_M-\frac{1}{8}{\textbf{H\,}^\prime}^2E_3^\flat\wedge E_4^\flat+\frac{1}{8}{\textbf{H\,}^\prime}^2d\theta\wedge d\varphi,$$
where $\flat$ denotes the $g$-Riemannian dual and ${\textbf{H\,}^\prime}={\textbf{H\,}^\prime}(t).$
\end{prop}
\begin{proof}
We consider the $J$-adapted local $g$-orthonormal frame $\{e_1,\ldots,e_6\}=\{E_1,E_2,E_3,E_4,T,JT\}.$ We have $$[E_4,T]=[E_4,JT]=\frac{\textbf{H}^\prime}{2}(T-JT).$$
Hence, $4N(E_3,T)=-\textbf{H}^\prime JT$ and $4N(E_4,T)=-\textbf{H}^\prime T$, where $N$ is the Nijenhuis tensor of $J.$ Using Corollary~\ref{2-chern-formula}, we have the following
\begin{align*}
r(E_i,E_j)-\rho^\nabla_M(E_i,E_j)&=\sum_{l,k=1}^{6}g\left(JE_i,N(e_l,e_k)      \right)g\left(E_j,N(e_l,e_k)      \right)\\
&+\sum_{l,k=1}^{6}g\left(N(E_i,e_k),e_l      \right)g\left(N(E_j,e_k),Je_l       \right),\\
&=\sum_{l,k=1}^{6}g\left(N(E_i,e_k),e_l      \right)g\left(N(E_j,e_k),Je_l       \right).
\end{align*}
In particular, $\left(r-\rho^\nabla_M\right)$ applied to the pairs $(E_1,E_2),(E_1,E_3),(E_1,E_4),(E_2,E_3)$ and $(E_2,E_4)$ vanishes. However,
\begin{align*}
r(E_3,E_4)-\rho^\nabla_M(E_3,E_4)&=-2g\left(N(E_3,T),JT      \right)g\left(N(E_4,T),T       \right),\\
&=-\frac{1}{8}{\textbf{H}^\prime}^2.
\end{align*}
Similarly, using Corollary~\ref{2-chern-formula}, $\left(r-\rho^\nabla_M\right)$ applied to $(E_i,T)$ and $(E_i,JT)$ vanishes. Finally,  
\begin{align*}
r(T,JT)&=\sum_{l,k=1}^{6}g\left(JT,N(e_l,e_k)      \right)g\left(JT,N(e_l,e_k)      \right)+\sum_{l,k=1}^{6}g\left(N(T,e_k),e_l      \right)g\left(N(JT,e_k),Je_l       \right),\\
&=\sum_{l,k=1}^{6}g\left(JT,N(e_l,e_k)      \right)^2  -\sum_{l,k=1}^{6}g\left(N(T,e_k),e_l      \right)^2,\\
&=\frac{1}{8}{\textbf{H}^\prime}^2.
\end{align*}
\end{proof}

The form $\rho^\nabla_M$ was computed in~\cite{MR4173158} (see also~\cite{Angella:2020vt}). Hence, from Proposition~\ref{2-chern-cohomogeneity} we get that
on $(\mathbb{S}^1\times\mathbb{S}^1\times M,g,J)$ the second-Chern--Ricci form is
\begin{equation}\label{formula-dim-6}
r=\frac{4-4{h^\prime}^2-2hh^{\prime\prime}}{h^2}E_1^\flat\wedge E_2^\flat-\left(\frac{5h^\prime h^{\prime\prime}+hh^{\prime\prime\prime}}{hh^\prime}+\frac{1}{8}{\textbf{H}^\prime}^2\right)E_3^\flat\wedge E_4^\flat+\frac{1}{8}{\textbf{H}^\prime}^2d\theta\wedge d\varphi.
\end{equation}
We would like to find second-Chern--Einstein almost-K\"ahler metrics on $(\mathbb{S}^1\times\mathbb{S}^1\times M,g,J,\omega)$
where $$\omega=E_1^\flat\wedge E_2^\flat+E_3^\flat\wedge E_4^\flat+d\theta\wedge d\varphi.$$
Then, it follows from~(\ref{formula-dim-6}) that we need to solve
\begin{equation}\label{main-equation}
-\frac{h^{\prime\prime\prime}}{h^{\prime}}-\frac{h^{\prime\prime}}{h}+8\frac{    { h^{\prime}}^2}       {h^2}-\frac{8}{h^2}=0,
\end{equation}
where the function $h=h(t)$ satisfies the boundary conditions~(\ref{conditions}) with 
\begin{equation}\label{conditon2}
\frac{4-4{h^\prime}^2-2hh^{\prime\prime}}{h^2}=-\frac{5h^\prime h^{\prime\prime}+hh^{\prime\prime\prime}}{2hh^\prime}=\frac{1}{8}{\textbf{H}^\prime}^2\geq 0.
\end{equation}
We introduce the function $y=y(h)$ defined by $h^\prime=\sqrt{y(h)}.$ Then, $h^{\prime\prime}=\frac{1}{2}y^\prime(h)$
and $h^{\prime\prime\prime}=\frac{1}{2}y^{\prime\prime}(h)\,h^\prime$ (here $y^\prime(h)$ is the derivative with respect to $h$ etc). Then, Equation~(\ref{main-equation}) becomes
$$-\frac{1}{2}y^{\prime\prime}-\frac{1}{2h}y^{\prime}+\frac{8}{h^2}y=\frac{8}{h^2}.$$
The solution is given by $$y(h)=c_1h^4+\frac{c_2}{h^4}+1,$$
for some constants $c_1,c_2.$
Denote by $h(0)=h_0>0$ and $h(l)=h_l>0$. To satisfy the boundary conditions~(\ref{conditions}), we need a solution $y(h)$
such that $$y(h_0)=y(h_l)=0, y^\prime(h_0)=\frac{2}{h_0}, y^\prime(h_l)=-\frac{2}{h_l}.$$
Then,
\begin{align}
c_1&=-\frac{1}{h_0^4+h_l^4},\nonumber\\
c_2&=-\frac{h_0^4h_l^4}{h_0^4+h_l^4},\nonumber\\
h_0^4+h_l^4&=-{2(h_0^4-h_l^4)}.\label{m-equation}
\end{align}
Define $x=\frac{h_l}{h_0}.$ Then, Equation~(\ref{m-equation}) is equivalent to
$$P(x)=-x^4+3=0.$$
Remark that $P(1)=2>0$, and so we get the existence of a solution of $P(x)=0$ with $x=\frac{h_l}{h_0}>1.$ Thus,
there are solutions of~(\ref{m-equation}) making $h$ a positive and an increasing function on the interval $(0,l)$
and satisfying the boundary conditions~(\ref{conditions}). We also remark that the function $y=y(h)$ is a positive on the interval $(h_0,h_l).$ Explicitly, $$y(h)=-\frac{1}{4h_0^4}h^4-\frac{3h_0^4}{4}h^{-4}+1,$$
where $h_0>0.$

%{{\color{red}{include maple graph for a value of $h_0$}}

Moreover, the condition~(\ref{conditon2}) is also satisfied for any $h_0>0$ (actually $\textbf{H}^\prime(t)$ is nowhere zero on the interval $[0,l]$). 

%{{\color{red}{include maple graph for condition  for a value of $h_0$}}.

We can then define the function $t(h)$ on the interval $(0,l)$ to be $$t(h)=\int_{h_0}^h\frac{dh}{\sqrt{y(h)}}.$$
The function $h(t)$ is then the inverse of $t(h)$. We deduce then the following
\begin{thm}\label{hirzebruch}
Let $M$ be the first-Hirzebruch surface. Then, there exists on $\mathbb{S}^1\times\mathbb{S}^1\times M$
an infinite family of non-K\"ahler almost-K\"ahler second-Chern--Einstein metrics of the form~(\ref{form-dual}) dual to a general plane-fronted wave Lorentzian metric, of positive (non-constant) Hermitian scalar curvature. 
\end{thm}

\bibliographystyle{abbrv}%\bibliographystyle{ieeetr}

\bibliography{canonical_almost-kahler}

\end{document}